\documentclass{elsarticle}
\usepackage{amsthm,amsmath,amssymb,amsfonts,enumerate}
\newtheorem{theorem}{Theorem}[section]
\newtheorem{lemma}[theorem]{Lemma}
\newtheorem{proposition}[theorem]{Proposition}

\newtheorem{question}[theorem]{Question}

\theoremstyle{definition}

\theoremstyle{remark}
\newtheorem{remark}[theorem]{Remark}

\numberwithin{equation}{section}

\begin{document}

\begin{frontmatter}
\title{ Differential polynomial rings over locally nilpotent rings need not be Jacobson radical }

\author{Agata Smoktunowicz}
\address{Maxwell Institute for Mathematical Sciences, School of Mathematics, University of Edinburgh,\\ JCM Building, King’s Buildings, Mayfield Road
Edinburgh EH9 3JZ, Scotland, UK,\\   E-mail: A.Smoktunowicz@ed.ac.uk}

\author{Micha{\l} Ziembowski}
\address{Faculty of Mathematics and Information Science, Warsaw University of Technology,\\  00-662 Warsaw, Poland,\\ m.ziembowski@mini.pw.edu.pl}
\begin{abstract} 
We answer a question by Shestakov on the Jacobson radical in differential polynomial rings. We show that if $R$ is a locally nilpotent ring with a derivation $D$ then 
$R[X; D]$ need not be Jacobson radical. 
We also show that $J(R[X; D])\cap R$ is a nil ideal of $R$ in the case where $D$ is a locally nilpotent derivation and $R$ is an algebra over an uncountable field.
\end{abstract}
\end{frontmatter}

\section{Introduction}

Let $D$ be a derivation on a ring $R$. 
We recall that the differential polynomial ring $R[X; D]$ consists of
all polynomials of the form $a_nX^n + \cdots + a_1X + a_0$, where $a_i\in R$
for $i = 0, 1, \dots, n$.
The ring $R[X; D]$ is considered with pointwise addition, and 
multiplication given by $X^iX^j = X^{i+j}$ and $Xa = aX + D(a)$, for all $a\in R$.

In a seminal paper \cite{Amitsur}, S. A. Amitsur proved that the Jacobson radical $J(R[X])$ of  the ring of all polynomials 
in a commutative indeterminate $X$ over R is the polynomial ring over the nil ideal $J(R[X])\cap R$. 
Then in \cite{Jordan}, D. A. Jordan showed that if $R$ is a right Noetherian ring with identity then  $J(R[X; D])\cap R$ is a nil ideal of $R$.
Also, in \cite{FKM} M. Ferrero et al.
revealed that $J(R[X; D]) = (J(R[X; D])\cap R)[X; D]$, and that in the case where $R$ is commutative $J(R[X; D]) \cap R$ is also nil. Papers \cite{FK} and more recently \cite{Tsai} and \cite{GB}  provide further interesting results.

 Overall it is an open question as to whether $J(R[X; D])\cap R$ is nil, however in this  paper we show that $J(R[X; D])\cap R$ is nil if $R$ is an algebra over an uncountable field and $D$ is a locally nilpotent derivation.

At the 2011 conference held in Coimbra entitled ''Non-Associative Algebras and Related Topics'', I. P. Shestakov  asked the following interesting question concerning the Jacobson radical of  differential polynomial rings. 
\begin{question}(Shestakov)
 Let $R$ be a locally nilpotent ring with a derivation $D$ and let $S=R[X; D]$ be the differential polynomial ring. Is the Jacobson radical of $S$ equal to $S$?
\end{question}
Although Shestakov's question is related to Lie algebras, in this paper we only concentrate on solving the above problem.
Using the aforementioned result by Ferrero et al. and \cite[Theorem 3.3]{FKM}, the answer to Shestakov's question is affirmative, if $R$ is a commutative ring over a field of characteristic zero. On the other hand, we show that in general the answer is in the negative. 
We leave as an interesting open problem the question of whether  $R[X; D]$ is Jacobson radical in the case where $R$ additionally satisfies a polynomial identity.

\section{On the Jacobson radical of the differential polynomial rings}

In this section we show that $J(R[X; D])\cap R$ is nil if $R$ is an algebra over an uncountable field and $D$ is a locally nilpotent derivation.

\begin{proposition}
 Let $R$ be an algebra over an uncountable field, $D$ be a locally nilpotent derivation on $R$, and $R[X; D]$ be the differential polynomial ring. Then $J(R[X; D])\cap R$ is nil.
\end{proposition}
\begin{proof}
Let $r\in R$, we denote $s(r)=n$ if $D^{n}(r)=0$ and $D^{n-1}(r)\neq 0$. 
Define  $S$ to be the set of all series  $\sum_{i=0}^{\infty}c_{i}X^{i}$ with the property that for every natural number $\alpha $ there exist $n_{\alpha }$, such that for all $i>n_{\alpha }$ we have $s(c_{i})+\alpha <i$.

It is easy to prove that the set $S$ with  addition and multiplication the same as in $R[X; D]$ is a ring and $R[X; D]$ is a subring of $S$.
% (multiplication of two elements from $S$ is well defined, and product of two elements from $S$ is in $S$).

Let $c\in R\cap J(R[X; D])$, and let $p>s(c)$. The element $1-cX^{p}$ has inverse in $(R[X; D])^{1}$  (since $c\in J(R[X; D])$) and on the other hand it has inverse in  $S^{1}$ because $f=\sum_{i=1}^{\infty }(cX^{p})^{i}$ is in $S$, and $(1-cX^{p})(1+f)=(1+f)(1-cX^{p})=1$. It follows that these two inverses are equal, so $f\in R[X; D]$, hence $f=\sum_{i=1}^{n-1}z_{i}X^{i}$, for some $n$. 
 
Since the base field is uncountable,  it follows that for infinitely many $\alpha \in K$, inverse of element $1-\alpha cX^{p}$ is equal to $1+f_{\alpha }$ where
$f_{\alpha }=\sum_{i=1}^{\infty }(\alpha cX^{p})^{i}$, is in $R+RX+\ldots +RX^{n-1}$. 
It follows that the coefficient at $X^{np}$ is zero for infinitely many $\alpha $.

For $a\in R[X; D]$ let $(a)_k$ denote the coefficient at $X^{k}$ in $a$.

Observe  that  the coefficient at $X^{np}$ in $f_{\alpha }$ is equal to 
\[((\alpha cx^{p})^{n})_{np}+((\alpha cx^{p})^{n+1})_{np}+ \ldots +((\alpha cx^{p})^{m})_{np}\]
 for some $m$ (such $m$ exists because $f_{\alpha }$ is in $S$).  We can take $\alpha $ outside the bracket to get 
\begin{equation}\label{locallyder}
\alpha^{n} ((cx^{p})^{n})_{np}+\alpha^{n+1}(( cx^{p})^{n+1})_{np}+\ldots +\alpha^{m}(( cx^{p})^{m})_{np}=0.
\end{equation}
Next time using the fact that $R$ is an algebra over an uncountable field, we can see that  (\ref{locallyder}) is true for infinitely many $\alpha$, hence by using the Vandermonde matrix argument,  we get that 
$((cx^{p})^{n})_{np}=(( cx^{p})^{n+1})_{np}=\ldots = ((cx^{p})^{m})_{np}=0$, it follows that $c^{n}=((cx^{p})^{n})_{np}=0$, as required.
\end{proof}

\section{Shestakov's question}

In this section we solve Shestakov's question. We first introduce some notation.

Let $K$ be a field, and let $A$ be
a free algebra over $K$ with a countable set of free generators $\mathcal{X} = \{x_0, x_1, x_2, \ldots\}$. Obviously, the monomials of the form $x_{i_1}x_{i_2}\dots x_{i_n}$
where $i_1, \dots, i_n$ are 
non-negative integers, form a $K$-basis of $A$.
By $A^1$ we denote 
the algebra obtained from $A$ by the adjunction of a unity.
If we consider  a monomial $s = x_{i_1}x_{i_2}\dots x_{i_n}$ then $l(s)$ stands for 
length of $s$, $deg(s) = i_1 + \ldots + i_n$, and for  $q = 1, \dots, n$ by $s[q]$ we denote the element $x_{i_q}$.
If an element $a\in A$ is a sum  of monomials of the same degree multiplied by coefficients, then by $deg(a)$ we mean the number expressing the common degree of the monomials.
For a positive integer $n$ and a subset $S$ of $A$ by $S(n)$ we denote the set of all elements of $S$
which are sums of monomials of length equal to $n$ multiplied by coefficients. Finally, by $\mathcal{M}$ we denote the set of all monomials of $A$.

Consider the  $K$-linear map $D: A \to A$ such that for any $i$, $D(x_i) = x_{i+1}$, and for $a, b\in A,$ $$D(ab)=D(a)b + aD(b),\, D(a+b)=D(a) + D(b).$$ 
This is obvious that $D$ is a derivation on $A$.

For $k > 0$ we set  $\mathcal{X}_k=\{x_0, x_1, \dots, x_{k-1}\}$, and recursively we define the following subsets of $A$
\begin{equation}
W(k, n, 0) = \{x_{i_1}x_{i_2}\cdot\,\ldots\,\cdot x_{i_n}: x_{i_j}\in \mathcal{X}_k \text{ for all } j\},
\end{equation}
$$\text{ if } W(k, n, l) \text{ is defined, } W(k, n, l+1) = \{D(x): x\in W(k, n, l)\},$$
and finally 
\begin{equation}
W(k, n):= \bigcup_{t\geq 0}W(k, n, t).
\end{equation}
Obviously $W(k, n)$ is closed under derivation $D$.

We define for any positive integer $k$ the ideal $I_k$ of $A$ generated 
by $W(k, 2\cdot 100^{k^2})$, and the ideal $I = \sum_{k>0}I_k$ of $A$.

For any positive  integer $k$ we define the linear space 
\begin{equation}
W_k = \sum_{m = 0}^{\infty}A(m\cdot 100^{k^2})W(k, 100^{k^2})A^1.
\end{equation}

In the next part of our construction we would like to prove the following.

\begin{lemma}\label{1}
For any $k\geq 1$ we have $I_k\subseteq W_k$.
\end{lemma}

\begin{proof}
Recall that the ideal $I_k$ is generated by $W(k, 2\cdot 100^{k^2})$.
As $W_k$ is a linear space and right ideal of $A$  
to prove the lemma it is enough to show that for any monomial 
$v\in A$ and  $w\in W(k, 2\cdot 100^{k^2})$ we have 
$$w\in W_k \text{ and } vw\in W_k.$$

Firstly, we want to show that 
for any monomial 
$v\in A$ and  $w\in W(k, 2\cdot 100^{k^2})$,
$vw\in W_k$. 

Let $l(v) = p\cdot 100^{k^2} + q$
with $q < 100^{k^2}$. Then $v = v_1v_2$ for some  $v_1, v_2\in\mathcal{M}$ such that
$l(v_1) = p\cdot 100^{k^2}$ and $l(v_2) = q.$

As $W(k, 2\cdot 100^{k^2}) = \bigcup_{t\geq 0} W(k, 2\cdot 100^{k^2}, t)$ there exists $l$ such that $w\in W(k, 2\cdot 100^{k^2}, l)$. Thus
there exists $u\in  W(k, 2\cdot 100^{k^2}, 0)$ such that
$w = D^l(u)$. 
Moreover, $u = u_1u_2u_3$ with $$u_1\in W(k, 100^{k^2} - q, 0), 
u_2\in W(k, 100^{k^2}, 0), \text{ and } u_3\in W(k, q, 0).$$
Thus for some positive integers $\alpha_{(l_1, l_2, l_3)}$
$$D^l(u) = D^l(u_1u_2u_3) = \sum_{l_1 + l_2 + l_3 = l}\alpha_{(l_1, l_2, l_3)} D^{l_1}(u_1)D^{l_2}(u_2)D^{l_3}(u_3).$$
Now $$vw = vD^l(u) = \sum_{l_1 + l_2 + l_3 = n}\alpha_{(l_1, l_2, l_3)} vD^{l_1}(u_1)D^{l_2}(u_2)D^{l_3}(u_3).$$
Since for any $l_1$, $vD^{l_1}(u_1)$ is a sum of monomials of length $p\cdot 100^{k^2}$ and $D^{l_2}(u_2)\in W(k, 100^{k^2})$ we deduce that $vw\in W_k$.

In a similar way we can show that if $w\in W(k, 2\cdot 100^{k^2})$ then $w\in W_k$.
\end{proof}

Now, we come to the very crucial point of our construction. Namely,  for any positive integer $k$ we 
fix numbers 
$$c_{1} = 100^{(k-1)^2}, c_{2}=3\cdot 100^{(k-1)^2}, \ldots, c_{k+1} = 3^k\cdot 100^{(k-1)^2}$$
and define the set $Z_k$
which consists of all elements $a$ of $A$ which satisfy one of the following conditions:
\begin{enumerate}
\item $a = \kappa s$ where $\kappa\in K$, and $s\in\mathcal{M}$ is  such that  $l(s) = 100^{k^2} - 1$, and there exist non-negative integers $p < q \leq k$ such that 
$$s[3^p\cdot 100^{(k-1)^2}] = s[3^q\cdot 100^{(k-1)^2}].$$

\item $a = \kappa(s_1 + s_2)$ where $\kappa\in K$, $s_1, s_2\in\mathcal{M}(100^{k^2} - 1)$, and there exist non-negative integers
$p < q \leq n$ and $l_1 > l_2 > 0$ such that 
$$s_1[3^p\cdot 100^{(k-1)^2}] = x_{l_1},\, s_1[3^q\cdot 100^{(k-1)^2}] = x_{l_2}$$
$$s_2[3^p\cdot 100^{(k-1)^2}] = x_{l_2},\, s_2[3^q\cdot 100^{(k-1)^2}] = x_{l_1},$$
and $s_1[j] = s_2[j]$ for any $j\neq 3^p\cdot 100^{(k-1)^2}, 3^q\cdot 100^{(k-1)^2}$.
\end{enumerate}

We leave to the reader  verification of the following.

\begin{lemma}\label{lreader}
For any $k> 0$ and $a\in Z_k$, $D(a)$ is a sum of elements of $Z_k$.
\end{lemma}

For any positive  integer $k$ we define the linear space 
\begin{equation}\label{sk}
B_k = \sum_{m = 0}^{\infty}A(m\cdot 100^{k^2})Z_kA^1.
\end{equation}

\begin{remark}\label{remark}
Using  Lemma \ref{lreader} it is easy to see that for any $k$ the linear space  $B_k$ is closed under derivation $D$ (i.e. $D(B_k)\subseteq B_k$). 
Moreover, it is obvious that $B_k$ is right ideal of $A$, and $A(m\cdot 100^{k^2})B_k \subseteq B_k$ for any $m\geq 0$.
\end{remark}

\begin{lemma}\label{ooo}
For any $k\geq 1$ we have $I_k\subseteq B_k$.
\end{lemma}  

\begin{proof}
Let $k$ be a positive integer.

Using Lemma \ref{1} it is enough to show that each element  $w\in W(k, 100^{k^2})$ is a sum of elements of $Z_kA(1)$.

Consider an element $u\in W(k, 100^{k^2}, 0)$.
By construction $l(u) = 100^{k^2}$ and $u[j] \in  \mathcal{X}_k = \{x_0, \dots, x_{k-1}\}$
for any $j = 1,\dots, 100^{k^2}$. As $\lvert \mathcal{X}_k\rvert = k$ considering 
the sequence of $k+1$ elements
$$u[3^0\cdot 100^{(k-1)^2}], \ldots, u[3^k\cdot 100^{(k-1)^2}]$$
we deduce that there exist $0 \leq p < q \leq k$ such that $u[3^p\cdot 100^{(k-1)^2}] = u[3^q\cdot 100^{(k-1)^2}]$. Thus $u\in Z_k$ and we have proved that 
$W(k, 100^{k^2}, 0)\subseteq Z_kA(1)$. 

Fix a positive integer $l$ and consider
$w\in W(k, 100^{k^2}, l)$. Then there exists a monomial $u\in W(k, 100^{k^2}, 0)\subseteq Z_kA(1)$ such that $w = D^l(u)$. Thus $w$
is a linear combination of elements of $Z_kA(1)$ by Lemma \ref{lreader}.
\end{proof}

\begin{lemma}\label{important}
Let $a_1, a_2, \ldots, a_n$ be elements of $A$ such that for any $i$, $a_i\in A(100^{k^2}p_i - 1)$ for some $p_i>0$.
 Furthermore, assume that for any $i$,  $a _i\notin B_1 + \ldots + B_{k}$.  Then for any non-negative integers $m_1, m_2, \ldots, m_{n}$, 
$a_1x_{m_1}a_2x_{m_2}\cdot\ldots\cdot x_{m_{n-1}}a_nx_{m_n}\notin B_1 + \ldots + B_{k}$.
\end{lemma}

\begin{proof} Denote $\xi _{i}=100^{k^{2}}p_{i}-1$ for $i=1,2, \ldots ,n$ and  set $B = B_1 + \ldots +B_{k}$. 
 For each $i\leq n$ fix  a linear map $\varphi _{i}: A(\xi_{i})\to  A(\xi _{i})$ such that $Ker(\varphi _{i}) = B\cap  A(\xi _{i})$.
 Denote $\xi =n+\sum_{i=1}^{n}\xi _{i}$. We will define a maping $\phi  :A(\xi )\to A(\xi)$ first for monomials and then extend by linearity to all elements of $A(\xi)$ in the following way. Let $v=\prod_{i=1}^{n}v_{i}u_{i}$ with $v_{i}\in \mathcal{M}(\xi _{i})$, $u_{i}\in \mathcal{X}$, then define 
$\phi (v)=\prod_{i=1}^{n}\varphi _{i}(v_{i})u_{i}$.  
 Observe that even thought $a_{i}$ may not be monomials, by linearity we get: 
$\phi (a_1x_{m_1}a_2x_{m_2}\cdot\ldots\cdot x_{m_{n-1}}a_nx_{m_n})=\prod_{i=1}^{n}\varphi _{i}(a_{i})x_{m_i}$. 

Observe now that if $w\in  A(\xi )\cap B$, then $\phi (w)=0$. It follows because by the definition of sets $B_{1}, \ldots ,B_{k}$  any element  from $B_{1}+\ldots +B_{k}$ is a linear combination of elements of the form $cv$ or
 $ucv$ or $uc$
where $u\in A(\xi _{1}+\ldots +\xi _{i-1}+i-1)$, $c\in B(\xi _{i})$,  for some $i$,  and where $u$ and $v$ are monomials. Then $\phi (ucv)\in A(\xi _{0}+\ldots +\xi _{i-1}+i-1)\varphi _{i} (c)A$. Observe that $\varphi _{i} (c)=0$, since $c\in B$, hence $\phi (ucv)=0$ as required . 
A similar argument gives us $\phi (uc)= 0 = \phi (cv)$.

To get a contradiction, suppose that $a=a_1x_{m_1}a_2x_{m_2}\cdot\ldots\cdot x_{m_{n-1}}a_nx_{m_n}\in B$ for some $m_1, \ldots, m_{n}$.
Then
$\phi (a) = \varphi _{1}(a_1)x_{m_1}\varphi _{2}(a_2)x_{m_2}\cdot\ldots\cdot x_{m_{n-1}}\varphi _{n}(a_n)x_{m_n}$.
On the other hand $\phi (a)=0$ since $a\in B$.
It follows that $\varphi _{1}(a_1)x_{m_1}\varphi _{2}(a_2)x_{m_2}\cdot\ldots\cdot x_{m_{n-1}}\varphi _{n}(a_n)x_{m_n}=0$, and hence 
$\varphi _{i}(a_{i})=0$ for some $i$, so $a_{i}\in Ker(\varphi_{i})\subseteq  B$, as required.
\end{proof}

\begin{lemma}\label{22}
Let $m \geq 0$ and $(x_0X)^m = a_mX^m +  a_{m-1}X^{m-1} + \ldots + a_0$. If for 
$s = x_{n_1}x_{n_2}\dots x_{n_m}$ and $\kappa\in K$, $\kappa s$ is a summand of a coefficient $a_t$ for some $t \geq 0$, then $t=m-deg(s)$ and for any $i = 1, \ldots, m$, $\sum_{j=1}^i n_j\leq i-1$.  
\end{lemma}

\begin{proof}
Assume that the claim is proved for  $m-1$ and let $\kappa s$ be  a summand of a coefficient $a_t$ for some $t \geq 0$, $\kappa\in K$ and $s=  x_{n_1}x_{n_2}\dots x_{n_m}$.  
It is easy to see that it must be that for some positive integer $\kappa_1$, $\kappa_1 x_{n_1}x_{n_2}\dots x_{n_{m-1}}$ is a summand of a coefficient of $(x_0X)^{m-1}$. Thus by our assumption  for $q= x_{n_1}x_{n_2}\dots x_{n_{m-1}}$,
$\kappa_1q$ is a summand of a coefficient at $X^{m - 1 - deg(q)}$, and 
for any $i = 1, \ldots, m-1$, $\sum_{j=1}^i n_j\leq i-1$.
Based on the above $\kappa_1\cdot {{m-1-deg(q)}\choose {n_m}}x_{n_1}x_{n_2}\dots x_{n_m}$ is a summand 
of a coefficient of 
\begin{equation}\label{eq1}
(\kappa_1 x_{n_1}x_{n_2}\dots x_{n_{m-1}}X^{m - 1 - deg(q)})(x_0X),
\end{equation}
and $\kappa = \kappa_1\cdot {{m-1-deg(q)}\choose {n_m}}$.
Thus $n_m \leq m - 1 - deg(q)$, so 
$\sum_{j\leq m}n_j \leq m-1$. Moreover, using our assumption and  (\ref{eq1}) we get $t = m-1 - deg(q) + 1 - m_n = t - deg(s)$. Thus the proof is  complete.
\end{proof}

\begin{lemma}\label{nextlemma}
If $m+1=100\cdot h$  for some positive integer $h$, and $(x_0X)^{m}=\sum_{i=0}^{m}a_{i}X^{i}$ with $a_{i}\in A(100\cdot h - 1)$, then 
there is $i>{3\over 4}(m+1)$ such that $a_{i}\notin B_{1}$ and $a_{i+1}, a_{i+2}, \ldots \in B_{1}$.
\end{lemma}

\begin{proof}
As a base case we consider $m = 99$ ($h=1$). 
It is not hard to see that  as a summand of
$a_{97}$ we have $s\in A(99)$ where $s$ is a monomial such that  $s[3] = x_2$  and $s[j] = x_0$ for all
$j\neq 3$. Moreover, by Lemma \ref{22}(ii) the other summands $\kappa v$ ($\kappa\in K$ and $v$ is a monomial) of $a_{97}$ satisfy $v[1] = x_0$, $v[3] \in \{x_0, x_1\}$. Thus considering the  construction of the set $Z_1\subseteq B_1$ we deduce that $a_{97}\notin B_1$. Since fact that 
$a_{99} = x_0^{99}\in B_1$ is clear, taking $i = 97$ if $a_{98}\in B_1$, and $i = 98$ if $a_{98}\notin B_1$ we finish the proof in this case.

Assume that our claim is true for all positive integers smaller than $h$ and consider
$m$ such that $m+1 = 100\cdot h$, and 
$(x_0X)^{m}=\sum_{n=0}^{m}a_{n}X^{n}$ with $a_{n}\in A(100\cdot h - 1)$. 
By assumption on $h$ and the first part of the proof, there exists a coefficient $b_{i_1}$ of $(x_0X)^{99}$ such that $i_1>{3\over 4}100$,
$b_{i_1}\notin B_1$ and $b_{i_1+1}, \ldots \in B_1$, and a coefficient $e_{i_2}$ of $(x_0X)^{100(h-1)-1}$ such that $i_2> {3\over 4}(100(h-1))$,
$e_{i_2}\notin B_1$ and $e_{i_2+1},\ldots \in B_1$.

Consider the coefficient $a_{i_1+i_2+1}$ of
$$(x_0X)^m = (x_0X)^{99}x_0X(x_0X)^{100(h-1)-1}.$$
Using above and Remark \ref{remark} it is easy to check that
$$a_{i_1 + i_2+1} = b_{i_1}x_0e_{i_2} + \sum_l\overline{b_l}D^{j_l}(x_0)\overline{e_l}$$ for some non-negative integers $l, j_l$ and $\overline{b_l}, \overline{e_l}$ such that for every $l$ either $\overline{b_l}\in B_1$ or $\overline{e_l}\in B_1$. Thus $\sum_l\overline{b_l}D^{j_l}(x_0)\overline{e_l}\in B_1$, and if $a_{i_1+ i_2+1}\in B_1$
we have $b_{i_1}x_0e_{i_2}\in B_1$. But this is impossible by Lemma \ref{important}. Thus  $a_{i_1+ i_2+1}\notin B_1$.
Moreover, $i_1+i_2 +1 >\frac{3}{4}100h$.
Taking  the biggest $i$ such that $a_i\notin B_1$ (obviously $i\geq i_1+i_2+1> \frac{3}{4}100h$) we get $a_{i+1}, \ldots \in B_1$.
Thus the claim is proved. 
\end{proof}

Recall that numbers $c_{1}, c_{2}, \ldots, c_{k+1}$ for a positive integer $k$ and constructed set $Z_{k}$  are fixed  and  $c_{1}=100^{(k-1)^2}, c_{2}=3\cdot 100^{(k-1)^2}, \ldots,\, c_{k+1}=3^k\cdot 100^{(k-1)^2}$. To simplify the notation we will from now on take
$c_0 = 0, c_{k+2} = 100^{k^2}$.

\begin{lemma}\label{important2}
 Let $a\in A(100^{k^{2}}-1)$ for some $k>0$.  Suppose that $a\in B_{1}+\ldots +B_{k}$. 
Denote $\xi _{i}=c_{i}-c_{i-1}-1$ for $i\leq k+2$.
Suppose that $\bar b$ is sum of all summands of $a$ which belong to  
$$E=A(\xi _{1})x_{c_{0}}A(\xi_{2})x_{c_{1}}A(\xi _{3})x_{c_{2}} \ldots x_{c_{k}}A(\xi _{k+2}).$$ Suppose moreover that for any permutation $\sigma $ of the elements $c_{0}, c_{1}, \ldots  , c_{k}$,  which is no the identity permutation, element $a$ has no summands 
 which belong to the set  $E^{\sigma }=A(\xi_{1})x_{\sigma (c_{0})}A(\xi _{2})x_{\sigma (c_{1})}A(\xi _{3})x_{\sigma (c_{2})} \ldots x_{\sigma (c_{k})}A(\xi _{k+2})$. Then $\bar b\in B_{1}+\ldots +B_{k-1}$.
\end{lemma}

\begin{proof} The proof follows from the fact that set $B_{k}$ acts on different elements than sets $B_{1}, \ldots , B_{k-1}$.  We  define a maping $\phi  : A(100^{k^{2}}-1 )\to A(100^{k^{2}}-1)$ first for monomials and then extend by linearity to all elements of $A(100^{k^{2}}-1)$ in the following way. Let $v=(\prod_{i=1}^{k+1}v_{i}u_{i})w$ with $v_{i}\in \mathcal{M}(\xi _{i})$, $u_{i}\in \mathcal{X}$, $w\in \mathcal{M}(\xi _{k+2})$ then define 
\begin{itemize}
\item[1.] $\phi (v)=v$ if $v\in E$,
\item[2.] $\phi (v)=(\prod_{i=1}^{k+1}v_{i}x_{c_{i-1}})w$ if $u_{1}, u_{2},\ldots , u_{k+1}$ is an even permutation of elements $x_{c_{0}}, x_{c_{1}}, \ldots , x_{c_{k}}$,
\item[3.] $\phi (v)=-(\prod_{i=1}^{k+1}v_{i}x_{c_{i-1}})w$  if $u_{1}, u_{2}, \ldots ,u_{k+1}$ is an odd permutation of elements $x_{c_{0}}, x_{c_{1}}, \ldots , x_{c_{k}}$,
\item[4.] $\phi (v)=0$ if  $u_{1}, u_{2},  \ldots , u_{k+1}$ is not a permutation of  elements $x_{c_{0}}, x_{c_{1}}, \ldots , x_{c_{k}}$.
\end{itemize}
 Recall that an even permutation is  obtained from an even number of two-element swaps. By the definition of set $Z_{k}$ it follows that $\phi (t)=0$ for every $t\in B_{k}\cap A(100^{k^{2}}-1)$ (this can be checked for all generating relations of $Z_{k}$).

Set $B = B_{1}+\ldots +B_{k-1}$.
 Observe now that if $p\in  A(100^{k^{2}}-1 )$ and $p\in B$, then $\phi (p)\in B$. It follows because by the definition of sets $B_{1}, \ldots ,B_{k-1}$  any element  from $B$ is a linear combination of elements of the form $cv$ or
$ucv$ or $uc$
 where $u\in \mathcal{M}(\xi _{1}+\ldots +\xi _{i-1}+i-1)$, $v\in \mathcal{M}$,  $c\in B\cap A(\xi _{i})$,  for some $i = 1, \ldots, k+2$.  
It follows that $\phi (p)\in B$, as required.
Consequently it follows that if $t\in (B_{1}+\ldots+B_{k-1} + B_{k})\cap A(100^{k^{2}}-1)$ then $\phi (t)\in B_{1}+\ldots +B_{k-1}$. 

By assumption $a\in B_{1}+\ldots +B_{k}$. This implies $\phi (a)\in B$. 
 Observe that $\phi (a)=\phi (\bar {b})$, as $a$ has no summands for any set $E^{\sigma }$, and $\phi (s)=0$ for any summand $s$ of $a$ which is not in $E$  or $E^{\sigma }$ for some permutation $\sigma $.
 Therefore, $\phi (\bar {b})\in B_{1}+\ldots +B_{k-1}$.
 But, by the definition of mapping $\phi $, $\phi (e)=e$ for any $e\in E$, and so  $\phi (\bar b)=\bar {b}$. Consequently,  
$\bar b\in B_{1}+\ldots +B_{k-1}$, as required.
\end{proof}

Keeping in mind that $c_0 = 0, c_{1}=100^{(k-1)^2}, c_{2}=3\cdot 100^{(k-1)^2}, \ldots,\, c_{k+1}=3^k\cdot 100^{(k-1)^2}, c_{k+2} = 100^{k^2}$  we will prove the following.

\begin{lemma}\label{main}
Let $k$ be a natural number.  If $m+1=100^{k^2}\cdot h$  for some positive integer $h$, and $(x_0X)^{m}=\sum_{j=0}^{m}a_{j}X^{j}$ with $a_{j}\in A(100^{k^2}\cdot h - 1)$ then 
there is $i>({1\over 2} + {1\over 2(k+1)})(m+1)$ such that $a_{i}\notin B_{1} + \ldots + B_{k}$ and $a_{i+1}, a_{i+2}, \ldots \in B_{1}+ \ldots + B_{k}$.
\end{lemma}

\begin{proof} Using Lemma \ref{nextlemma} we can assume that the claim is true for all positive integers smaller than $k$ and $k > 1$.

Let $h=1$. 
Consider 
$$(x_0X)^{100^{k^2}-1}=\big(\prod_{j=1}^{k+1}[(x_0X)^{c_i - c_{i-1}-1}x_0X]\big)\cdot (x_0X)^{c_{k+2} - c_{k+1}-1}.$$
 By inductive assumption for $i= 1, \ldots, k+2$
$$(x_0X)^{c_{i}- c_{i-1} - 1}=a_{\alpha_{i}}X^{\alpha_{i} }+f_{i}(X)+g_{i}(X)$$ where 
\begin{equation}\label{eqqq1}
\alpha_{i} > ({1\over 2} + {1\over 2k})(c_{i} - c_{i-1})
\end{equation}
 $a_{\alpha_{i}}\in A$, $a_{\alpha_{i}}\notin B_{1} + \ldots +B_{k-1}$, and  
$$f_{i}(X)\in \sum_{j=0}^{\alpha_{i} -1}AX^{j},\, g_i(X)\in ((B_{1}+\ldots + B_{k-1})X^{\alpha_{i}+1})[X].$$

Thus setting $F_i = a_{\alpha_{i}}X^{\alpha_{i} }+f_{i}(X)+g_{i}(X)$ for $i=1, \ldots, k+2$ we have 
\begin{equation}\label{ax}
(x_0X)^{100^{k^2}-1} = F_1x_0XF_2x_0X\cdots\ldots\cdot x_0XF_{k+2}.
\end{equation}

Set  
\begin{equation}\label{t}
t = 1 + \alpha_{k+2}.
\end{equation}
Straightforward computation shows that 
$t >  ({1\over 2} + {1\over 2(k+1)})100^{k^2}$.
Consider the coefficient $a_t$ of $(x_0X)^{m}=\sum_{j=0}^{m}a_{j}X^{j}$ and

{\setlength\arraycolsep{1pt}
\begin{eqnarray}
b & = & a_{\alpha_1}D^{\alpha _{1}}(x_0)a_{\alpha_2}D^{\alpha_2+1}(x_{0})a_{\alpha_3}D^{\alpha_3+1}(x_{0})\cdot\ldots \cdot a_{\alpha_{k+1}}D^{\alpha_{k+1}+1}(x_{0})a_{\alpha_{k+2}} 
\nonumber\\
& = & {}  a_{\alpha_1}x_{\alpha_1}a_{\alpha_2}x_{\alpha_2+1}\cdot\ldots \cdot a_{\alpha_{k+1}}x_{\alpha_{k+1}+1}a_{\alpha_{k+2}}
\end{eqnarray}}
which is a summand of $a_t$.   

We  consider the element $\overline{b}$ which is a sum of all summands of $a_t$ which are of the form 
$$q = q_1x_{\alpha _{1}}q_2x_{\alpha_2+1}q_3x_{\alpha_3+1}\cdot\ldots\cdot q_{k+1}x_{\alpha_{k+1}+1}q_{k+2}$$ where for every $i =1, \ldots, k+2$, $q_i= D^{l_i}(w_i)$ with  $l_i$  non-negative integers, and $w_i$ being coefficients
of $a_{\alpha_i}X^{\alpha_i} + f_i(X) + g_i(X)$ (notice that always $deg(w_i) \leq deg(q_i)$). 
Observe that since $q$ is a summand of $a_t$ using  Lemma \ref{22} we have
$$deg(q) = deg(q_1) + \ldots + deg(q_{k+1}) + deg(q_{k+2}) + \alpha_1 + \ldots+ \alpha_{k+1}+k= m-t.$$
Thus by $(\ref{t})$
\begin{equation}\label{eq2}
m + 1 = (deg(q_1)+1+\alpha_1)+\ldots+(deg(q_{k+1})+1+\alpha_{k+1}) + (deg(q_{k+2})+1+\alpha_{k+2}).
\end{equation}
We will need to see  $m+1\, (= 100^{k^2} = c_{k+2})$ also in the following form
\begin{equation}\label{eq3}
m + 1= (c_1 - c_0) + (c_2 - c_1) + \ldots + (c_{k+1} - c_k) + (c_{k+2}-c_{k+1}).
\end{equation}

Observe that if for some $i$, $w_i$ is a coefficient of $f_{i}(X) \in \sum_{j=0}^{\alpha_{i} -1}AX^{j}$, or $w_i = a_{\alpha_i}$ and $l_i>0$, then 
 $deg(q_i) \geq  c_i - c_{i-1} - 1 - (\alpha_i - 1) = c_i - c_{i-1} - \alpha_i$ which implies $deg(q_i) + \alpha_i + 1 > c_i - c_{i-1}$.
The last fact together with $(\ref{eq2})$ and $(\ref{eq3})$ imply that in the described situation
there exists $j$ such that $deg(q_j) + \alpha_j + 1<c_j - c_{j-1}$. So we have $deg(w_j) + \alpha_j + 1 < c_j - c_{j-1}$.
 But then by Lemma \ref{22},   $w_j$ is a coefficient of $g_j(X)$, and  using Remark
\ref{remark} we have $q\in B_1+\ldots +B_{k-1}$.

By above consideration we get 
$\overline{b} - b\in B_1+\ldots +B_{k-1}$. 
Thus $\overline{b}\notin B_1+\ldots + B_{k-1}$. Indeed, otherwise $b\in B_1+\ldots +B_{k-1}$ and by Lemma
\ref{important} (taking $k-1$ in place of $k$), $a_{\alpha_j}\in B_1+\ldots + B_{k-1}$ for some $j=1, \ldots, k+2$, a contradiction. 
As by (\ref{eqqq1}) for $i=1, \ldots, k$,
$\alpha_{i+1} \geq ({1\over 2} + {1\over 2k})\cdot 2\cdot 3^{i-1}100^{(k-1)^2} > 3^{i-1}100^{(k-1)^2}=c_i$,
using Lemma \ref{22} we can seet that the element $\bar b$ satisfies assumptions of Lemma \ref{important2} for $a=a_{t}$. 
But then $\overline{b}\notin B_1+\ldots + B_{k-1}$ implies $a=a_{t}\notin B_1+\ldots + B_{k}$.
As  $t >({1\over 2} + {1\over 2(k+1)})100^{k^2}$ taking 
the maximal $i$ such that $a_i\notin B_1+\ldots + B_k$ and $i > ({1\over 2} + {1\over 2(k+1)})100^{k^2}$ we get $a_{i+1}, \ldots \in B_1 + \ldots + B_k$, which finishes our argument.

\medskip 
 Now we begin the second part of the proof.
Suppose that the claim is true for $k$ and all positive integers smaller then $h$.
Then for $m = 100^{k^2}h - 1$ and $(x_0X)^{m}=\sum_{i=0}^{m}a_{i}X^{i}$ we have 

{\setlength\arraycolsep{1pt}
\begin{eqnarray}
(x_0X)^m & = & a_{\alpha_1}D^{\alpha _{1}}x_0a_{\alpha_2}D^{\alpha_2+1}(x_{0})a_{\alpha_3}D^{\alpha_3+1}(x_{0})\cdot\ldots \cdot a_{\alpha_{k+1}}D^{\alpha_{k+1}+1}(x_{0})a_{\alpha_{k+2}}
\nonumber\\
& = & {}  (b_{i_1}X^{i_1}+
f_{1}(X)+g_{1}(X))x_0X(e_{i_2}X^{i_2}+f_{2}(X)+g_{2}(X))
\end{eqnarray}}
where
$i_1 >({1\over 2} + {1\over 2(k+1)})100^{k^2}$, $b_{i_1}\in A$, 
$b_{i_1}\notin B_{1}+\ldots + B_{k}$, $f_{1}(X)\in \sum_{i=0}^{i_1 -1}AX^{i}$, $g_{1}(X)\in ((B_{1}+\ldots + B_{k})X^{i_1+1})[X]$, and
$i_2 >({1\over 2} + {1\over 2(k+1)})100^{k^2}(h-1)$, $e_{i_2}\in A$, $e_{i_2}\notin B_{1}+\ldots + B_{k}$, $f_{2}(X)\in \sum_{i=0}^{i_2 -1}AX^{i}$, $g_{2}\in ((B_{1}+\ldots + B_{k})X^{i_2+1})[X]$.

Consider the coefficient $a_{i_1 + i_2 + 1}$ of $(x_0X)^m$.
We can see that $a_{i_1 + i_2 + 1} = b_{i_1}x_0e_{i_2} + b$ for some 
$b\in B_{1}+\ldots + B_{k}$. Thus if $a_{i_1 + i_2 + 1}\in B_1 + \ldots +B_k$, then $b_{i_1}x_0e_{i_2}\in B_1 + \ldots + B_k$. But this contradicts Lemma \ref{important}. Thus $a_{i_1 + i_2 + 1}\notin B_1+\ldots+ B_k$.
Notice that 
$i_1 + i_2 + 1 > ({1\over 2} + {1\over 2(k+1)})100^{k^2}h$. Thus obviously we can find the maximal $i$ such that $a_i\notin B_1+\ldots + B_k$ and $i> ({1\over 2} + {1\over 2(k+1)})100^{k^2}h$. Then also $a_{i+1}, \ldots \in B_1 + \ldots + B_k$ and our  proof is complete.
\end{proof}

\begin{theorem}
There exists a locally nilpotent ring $R$ and a derivation $D$ on $R$ such that $R[X; D]$ is not a Jacobson radical ring.
\end{theorem}
\begin{proof}
As it is not hard to see that $D(I)\subseteq I$, we can consider the $K$-algebra $R  = A/I$ and the natural  derivation on $R$ induced by $D$ which we will denote also by $D$. It is easy to see that $R$ is a locally nilpotent algebra. By Lemma \ref{main}
for any positive integer $k$ there exists positive integer $i$ such that for $(x_0X)^{100^{k^2} - 1} = \sum_{j=0}^{100^{k^2} - 1}a_{j}X^{j}$,
$a_i\notin B_1 + \ldots + B_k$. Since any monomial which is a summand of an element of $B_n$ for $n>k$, has length at least $100^{(k+1)^2}$, using Lemma \ref{ooo}    we get  $a_i\notin I$.
Finally,  as $R[X; D]$ is graded by positive integers when we assign gradation $1$ to elements $x_{l}$ for any $l$, and gradation $0$ to $X$,  and in graded rings homogeneous quasiregular elements are nilpotent \cite{Smoktunowicz} we state that the ring $R[X; D]$ is not Jacobson radical.
\end{proof}
\begin{theorem} Let $K$ be a field. Then there  exists a locally nilpotent $K$-algebra $R$ and a derivation $D$ on $R$ such that $R[X; D]$ is not  a Jacobson radical ring.
\end{theorem}
\begin{proof}
 The proof is the same as the proof of the previous Theorem.
\end{proof}

\section*{Acknowledgements}
 The authors would like to give thanks to Andr\'e Leroy and Jerzy Matczuk for helpful suggestions concerning Section 2 related to Ferrero's question during their research visits to Lens. The authors are also very grateful to Ivan Shestakov for useful comments on the preprint of this paper. The research  of Agata Smoktunowicz was funded by ERC grant 320974.

\def\auth#1{{\sc #1}}
\def\titlart#1{{\it #1.}}
\def\titlj#1{{\rm #1\mbox{}}}
\def\vol#1{{\bf #1}:}
\def\no#1{{\rm no.\ #1,}}
\def\date#1{{\rm (#1).}}

\end{document}